\documentclass[12pt]{amsart}
\usepackage{amssymb}
\usepackage[T1]{fontenc}
\usepackage[bitstream-charter]{mathdesign}

\newtheorem{thm}{Theorem}[section]
\newtheorem{lem}[thm]{Lemma}
\newtheorem{prop}[thm]{Proposition}

\newtheorem{cor}[thm]{Corollary}

\newtheorem{remark}[thm]{Remark}

\begin{document}

\newcommand{\FN}{\mathbb{N}}
\newcommand{\FZ}{\mathbb{Z}}  % Integer ring Z
\newcommand{\FC}{\mathbb{C}}  % Complex field
\newcommand{\FR}{\mathbb{R}}  % Real field
\newcommand{\FQ}{\mathbb{Q}}  % Rational field

\newcommand{\fa}{\mathfrak{a}}
\renewcommand{\a}{\mathfrak{a}}
\newcommand{\fb}{\mathfrak{b}}
\newcommand{\fc}{\mathfrak{c}}
\newcommand{\fm}{\mathfrak{m}}
\newcommand{\fo}{\mathfrak{o}}
\newcommand{\fp}{\mathfrak{p}}
\newcommand{\fq}{\mathfrak{q}}

\title[partial zeta funtion]{Polynomial behavior of special values of partial zeta function of real quadratic fields at s=0}

\author{Byungheup Jun and Jungyun Lee}
\email{byungheup@gmail.com} \email{lee9311@kias.re.kr}
\address{School of Mathematics,  Korea  Institute for Advanced Study\\
Hoegiro 87, Dongdaemun-gu, Seoul 130-722, Korea}

\begin{abstract}
We compute the special values of partial zeta function at $s=0$ for family of real quadratic fields $K_n$ and ray class ideals $\fb_n$ such that $\fb_n^{-1} = [1,\delta(n)]$ where the continued fraction expansion of $\delta(n)$ is purely periodic and
each terms are polynomial in $n$ of bounded degree $d$.
With an additional assumptions, 
we prove that the special values of  partial zeta function at $s=0$  % of the ray class ideal $\fb_n$ 
behaves as quasi-polynomial. 
We apply this to obtain that the special values the Hecke's $L$-functions at $s=0$ for a family of for a Dirichlet character $\chi$ behave as  quasi-polynomial as well. We compute out explicitly the coefficients of the quasi-polynomials. 
Two examples satisfying the condition are presented and for these families the special values of the partial zeta functions at $s=0$. % for these families are explicitly computed.

\end{abstract}

\date{2011.5.23}
\maketitle
\tableofcontents

\section{Introduction}

This note is complementary to our previous work ``the behavior of
Hecke's L-function at s=0''(\cite{J-L}). In \cite{J-L}, a series of real
quadratic fields $K_n$ with fixed ideals $b_n^{-1} = [1,\delta(n)]$ and a mod-$q$ Dirichlet character
$\chi$, which yields a mod-$q$ ray class character $\chi_n:=\chi\circ N_{K_n/\FQ}$ for each $K_n$,
are considered. 
We gave a condition on $\delta(n)$ to ensure a controlled 
behavior of the special value at $s=0$ of the Hecke's partial $L$-function of $b_n$ 
twisted with a ray class character $\chi_n$ of  modulus $q$ in a certain way.
More precisely, we called this property `linearity of partial Hecke $L$-values' when
$$
L_{K_n}(0, \fb_n, \chi_n) = A(r) k + B(r)
$$
where $n=qk+r$ and for some constants $A(r),B(r)$ associated to $r=0, 1,\ldots,r-1$.
The coefficients $A(r), B(r)$ can be explicitly computed and are shown to be lying inside
$\frac{1}{12q^2} \FZ[\chi(1),\chi(2),\ldots,\chi(r-1)]$. 
%character at $s=0$ to be linear in $k$ with coefficients
%$A_\chi(r), B_\chi(r)$, where $n=qk+r$. 
Roughly written the
criterion is that the terms of the continued fractions of
$\delta(n)$ has to be linear function in $n$. 

The proto-typical result of this linearity
appeared first in Biro's proof of Yokoi's conjecture(cf.  \cite{Yokoi} for the conjecture and \cite{Biro1} for the proof). It was the
key ingredient with a class number one criterion in solving
Yokoi's conjecture.
Later  this moral has been extensively applied %in the same line of thought, 
in solving some class number one problems for some families of real quadratic fields when there
is an appropriate class number one criterion.% for a family of real
%quadratic fields. 
%The linearity of Hecke's L-values at $s=0$ has
%been equally used to solve class number one problem of other families of real quadratic fields.
(cf. \cite{Biro1},\cite{Biro2},\cite{Lee1},\cite{Lee2},\cite{Lee5},\cite{Lee4}).

In this paper, we deal with similar phenomenon with almost the same assumptions as in \cite{J-L}, but 
we have freed the linearity assumption on the terms of continued fraction expansion of $\delta(n)-1$.
%when the continued fraction expansion of $\delta(n)-1$ consists of polynomials in $n$ of degree bounded by an integer $d$. 
In this setting, the \textit{linearity} of the values of $\zeta$- or $L$-values in family is generalized to a `polynomial' expression . Written precisely, the special values as function in $n$ are taken in a packet of  polynomials in a periodic way. Functions of this type are called \textit{quasi-polynomials}. So in particular the `linearity' in \textit{loc.cit.} should read as quasi-linear function. 
A precise definition and a short discussion of quasi-polynomial
are presented 
%, we refer the reader to
in Section $2$ of this article.

Our main result is as follows:\\

\noindent\textbf{Theorem \ref{main}.} 
\textit{Let $\{K_n=\FQ(\sqrt{f(n)})\}_{n\in \FN}$ be a family of real quadratic fields where $f(n)$ is a positive square free integer for each $n$.
Suppose $\fb_n$ is an integral ideal relatively prime to $q$ such
that $\fb_n^{-1} = [1,\delta(n)]$. 
Assume   $\delta(n)-1$ has purely periodic continued fraction expansion  % is of the following form:
$$\delta(n)-1=[[a_0(n),a_1(n),\cdots,a_{s-1}(n)]]$$
of period $s$ 
%has a fixed length $s$ 
independent of $n$ and  $a_i(x)$ are polynomials of degree smaller than $d$. 
%\in \FZ[x]$. 
%Let $$d:=\max_{0\leq i\leq s-1 }\{deg(a_i(x))\}.$$
If $N_{K_n}(\fb_n(C+D\delta(n)))$ modulo $q$  is a function only depending on $C$, $D$ and $r$ for $n=qk+r$, %$r=0,1,\ldots,q-1$ (ie. determined by the residue of $n$ by $q$), %
%then
then 
$$
\zeta_q(0,\fb_n(C+D\delta(n))) = A_0(r) + A_1(r) n + \ldots + A_d(r) n^d
$$
where $n=qk+r$, $0\le r < q$ and
for some constants $A_0(r), A_1(r),\ldots, A_d(r)$ depending on $r$.
% is quasi-polynomial in the variable $n$ of degree $d$ and period $q$.
Furthermore, 
$
A_i(r) \in \frac{1}{12q^{i+2}}\FZ.
$}\\
%Moreover, the values of $i$-th coefficient functions are in $\frac{q^{i-2}}{12}\FZ$.}\\

A part of the result presented here was  announced  for the Hecke's $L$-values  in a family at the end of
\cite{J-L}.  
Actually, the consideration of the partial zeta function for ray class ideals instead of 
the partial Hecke's L-function of an ideal ameliorates the previous in two folds. Firstly, we give here an expression
of Hecke's $L$-function as twisted sum of the ray class partial zeta function:
\begin{equation*}
L_{K_n}(\chi,s,\fb) = \sum_{(C,D)\in \tilde{F'_{\delta}}}\chi((C+D\delta)\fb)\zeta_q(s,(C+D\delta)\fb)
\end{equation*}
(c.f. Proposition \ref{2.2}.).
If we restrict to the case $d=1$, using this identification, we can recover directly
the linearity of the values
of Hecke's $L$-functions at $s=0$(See Corollary \ref{2.5}.). For general $d$, we obtain the higher
degree generalization of the main result presented in \cite{J-L}. This is previously announced in \textit{loc.cit.}
Secondly, in principle the partial zeta function of a ray class
ideal contains finer information than the partial Hecke's $L$-function of the associated ideal. 

Our main idea is to develop and examine appropriate cone decomposition similar to Shintani and Zagier for partial zeta functions(cf. \cite{Shintani}, \cite{Zagier}). Once a simple cone is given, 
one can evaluate the Riemann sum for the partial zeta value at $s=0$ over the cone 
and it is simply expressed using
values of Bernoulli polynomial.  
For a ray class partial zeta values, we need
to sum over larger cone than that for ideal class  partial zeta or $L$ function. 
Unfortunately, the decompositions 
for a family of real quadratic fields with ideals are far from being uniform. 
But surprisingly again the cone decomposition for the associated ideal class behaves in a uniform way.
In particular, if we parameterize the orbit of the ray class ideal $\fb_n$ acted by the totally positive 
unit group  using a pair of integers $(C,D)$
such that $0\le C,D\le q-1$, this action has $q$-periodicity in the family parametrized by $n$. 

%We deal with two fold generalization
%of earlier work.
%First, we investigate the case when $\delta(n)-1$
%has continued fraction expansion such that each term is given by a
%polynomial of degree smaller than $N$ in $n$. Second, we control
%the behavior of  the ray class partial $\zeta$-values of the family $(K_n, b_n(C+D\delta(n)))$ instead of the Hecke's $L$-values with respect to a ray class character $\chi$ mod $q$. 
%The second part is a generalization in the sense that one can recover the linear behavior of 
%the Hecke's $L$-values in the specially chosen family(See Corollary \ref{2.5}).

%Thus the main theorem of \cite{J-L} can be recovered as a particular case of $d=1$ with a decomposition
%of partial Hecke's L-function into twisted sum of partial $\zeta$ function for ray class. 

This paper is composed as follows. In Section 2, we recall the notions of ray class partial zeta function and
some necessary stuffs. Then we state the main results. %relate this to partial Hecke's $L$-function
Section 3 is devoted to an expression of the partial zeta value at $0$, where we use the cone decomposition \textit{\`a la} Shintani and Zagier. In Section 4, we show that this domain decomposition 
is acted by the totally positive unit group modulo units congruent to $1$ modulo $q$, which has certain
invariance property in the family. This concludes the most important part of our main  theorem that
the special values behave in quasi polynomial for the family satisfying our assumption. 
In Section 5, we compute the coefficients of the quasi polynomials in the expression of partial zeta values 
to restrict possible coefficients. Finally, in Section 6, we compute the quasi polynomials for
two explicitly chosen families.  

\section{Decomposition of L-function into partial zeta function}

\subsection{Partial zeta function of a ray class}
For a number field $K$ and a fixed positive rational integer  $q$ as a conductor, 
the partial zeta function of the ray class of an ideal $\fa$ in $K$ is defined as
$$\zeta_q(s,\fa):=\sum_{\stackrel{\fc\sim_q\fa}{integral}}N(\fc)^{-s}$$
where $\fc\sim_q\fa$ means that $\fc=(\alpha)\fa$ for totally positive $\alpha\equiv 1\pmod{q}.$
For $\fc$ to be integral, $\alpha$ should be an element of $1+q\fa^{-1}$ and $\fc = \fa$ if and only
if $\alpha \in E_q^+$, where $E_q^+$ is the multiplicative group of totally positive units congruent to $1$ modulo $q$. Thus we have 
\begin{equation}\label{pz}
\zeta_q(s,\fa)=\sum_{\alpha\in(1+q\fa^{-1})^+/E_q^+} N(\fa \alpha)^{-s} 
 \end{equation}
This should not be confused with the partial zeta function for an ordinary ideal class. 
Note that this definition works not only for an ideal relatively prime to $q$ but for general ideal of $K$.

From now on, we assume $K$ to be a real quadratic field and 
we consider ray class partial zeta function of an integral ideal $\fb$ relatively prime to $q$ 
such that $\fb^{-1}= [1,\delta]$.
In this case, we have a description of the partial zeta function of other ray class than $\fb$
but in the same ideal class.  
% is an integral ideal relatively prime to $q$ such that $\fb^{-1}=[1,\delta]$.

\begin{lem}\label{par_zeta}
Suppose $C,D$ are integers such that $0\le C,D \le q-1$ and $((C+D\delta)\fb, q)=1$.
%$(C+D\delta)\fb$ is an integral ideal.
%For $(C,D)\in F_{\delta}$, 
Then we have
$$
\zeta_q(s,(C+D\delta)\fb)=\sum_{%\begin{subarray}{l}
 \alpha\in(\frac{C+D\delta}{q}+\fb^{-1})^+/E_q^+}
 %\\\alpha\gg 0\end{subarray}}
 N(q\fb\alpha)^{-s}.$$
\end{lem}

\begin{proof}
From (\ref{pz}) we have% above lemma, we have 
\begin{equation}\label{zeta1}
\zeta_q(s,(C+D\delta)\fb)=\sum_{\beta\in (1+ q(C+D\delta)^{-1}\fb^{-1})^+/E_q^+}N((C+D\delta)\fb\beta)^{-s}.
\end{equation}
put $\alpha=\beta\frac{C+D\delta}{q}$. This shows the equality. 
\end{proof}

Let us define 
$$
F := \{(C,D)\in \FZ^2| 0 \le C,D \le q-1, (C,D) \ne (0,0)\}
$$
and its subset
\begin{equation}\label{Funda}
F_{\delta}:=\{(C,D)\in \FZ^2 | 0\leq C,D\leq q-1, ((C+D\delta)\fb,q)=1\}.
\end{equation}

An element  $(C,D)$ of $F$ sends an ideal $\fa$ to another ideal $(C+D\delta)\fa$. 
If $(C,D)\in F_\delta$, $(C+D\delta)$ sends $\fb$ to another ideal  relatively prime to $q$.
% \fa$ is relatively prime to $q$, whenever so is $\fa$. 

%We shall denote 
The group $E^+$ of  totally positive units of $K$ acts on $F$ %$E^{+}$ action on $F_{\delta}$ 
by 
\begin{equation}\label{action}
 \epsilon\ast(C+D\delta)=C'+D'\delta 
\end{equation}
where
 $ (C+D\delta)\epsilon+q\fb^{-1}=C'+D'\delta+q\fb^{-1}\,\,\text{for}\,\, \epsilon\in E^{+}.$
Note that this action is inherited to $F_\delta$. 

Let 
$\tilde{F}' \subset F$ be a fundamental set for the quotient $F/E^+$ and 
$\tilde{F}_{\delta}' \subset \tilde{F}'$ be  a fundamental set for the quotient $F_{\delta}/E^+$ .

\subsection{Decomposition of partial Hecke L-function}
%by the action of $E^+$.
Now we consider the partial Hecke $L$-function for $(\fb,\chi)$ where $\chi$ is a ray class character of modulus $q$. Recall the partial Hecke $L$-function associated to $(\fb,\chi)$ is defined as follows:
$$
L_K(s,\chi,\fb) = \sum_{\fc \sim\fb} \frac{\chi(\fc)}{N(\fc)^{s}}
$$
where the summation runs over every integral ideal $\fc$ principally equivalent to $\fb$. 
%a ray class character $\chi$ of modulus $q$.

Keeping the notations introduced before, 
we obtain an expression of partial Hecke $L$-function of $(\fb,\chi)$ as sum of
ray class partial zeta functions of ray class ideals principally equivalent to $\fb$ twisted with values of
$\chi$. Since any ray class
associated to an ideal class of $\fb$ can be represented in the form $(C+D\delta)\fb$ and $E_q^+$ 
action on $(C,D)$ preserves the ray class, this sum is taken over $\tilde{F}'$. Moreover, as $\chi$ values $0$ 
at $(C+D\delta)\fb$ for $(C,D)$ outside $\tilde{F}'_\delta$, the summation actually 
runs over $\tilde{F}'_\delta$. 

Summarizing this discussion, we have
\begin{prop}
\label{2.2}
 Let $q$ be a positive integer. 
For an ideal $\fb\subset K$ relatively prime to $q$ and a ray class character $\chi$ modulo $q$, we have
\begin{equation*}
L(\chi,s,\fb) = \sum_{(C,D)\in \tilde{F'_{\delta}}}\chi((C+D\delta)\fb)\zeta_q(s,(C+D\delta)\fb).
\end{equation*}
\end{prop}

\begin{proof}
\begin{equation*}
\begin{split}
&L_K(s,\chi,\fb) =
\sum_{\begin{subarray}{c}\fa \sim \fb\\integral \end{subarray}}\chi(\fa)N(\fa)^{-s}\\
&= \sum_{(C,D)\in \tilde{F'_{\delta}}}\chi((C+D\delta)\fb)
 \sum_{%\begin{subarray}{l}
 \alpha\in(\frac{C+D\delta}{q}+\fb^{-1})^+/E_q^+}
 %\\\alpha\gg 0\end{subarray}}
 N(q\fb\alpha)^{-s}
\end{split}
\end{equation*}
Applying Lemma (\ref{par_zeta}), we have done the proof.
\end{proof}

%\begin{lem} \label{index} Let $q$ be a positive integer. 
%For an ideal $\fb\subset K$ relatively prime to $q$ and a ray class character $\chi$ modulo $q$, we have
%\begin{equation*}
%\begin{split}
%&L_K(s,\chi,\fb) =
%\sum_{\begin{subarray}{c}\fa \sim \fb\\integral \end{subarray}}\chi(\fa)N(\fa)^{-s}\\
%&= \sum_{(C,D)\in \tilde{F'_{\delta}}}\chi((C+D\delta)\fb)
 %\sum_{%\begin{subarray}{l}
 %\alpha\in(\frac{C+D\delta}{q}+\fb^{-1})^+/E_q^+}
 %\\\alpha\gg 0\end{subarray}}
 %N(q\fb\alpha)^{-s},
%\end{split}
%\end{equation*}
%where $\fa\sim \fb$ means that $\fb=\alpha\fa$ for some $\alpha\in K^*$. % relatively prime to $q$.
%\end{lem}

%\begin{proof}
%See Proposition 2.3 in \cite{J-L}.
%\end{proof}

%An immediate consequence of  Lemma 2.2 and  2.3 is the following: %,  we immediately have the following:

\subsection{Continued fractions} %and quasi polynomials}
To state our main result properly, we need to recall the notions of continued fractions and quasi polynomials. 

Let $[[a_0,a_1,\ldots,a_n]]$ be the purely periodic continued fraction
	$$[a_0,a_1,a_2,\ldots,a_n,a_0,a_1,\ldots],$$
	where
	$$
	[a_0,a_1,a_2,\ldots]:= a_0 + \cfrac{1}{a_1 +\cfrac{1}{a_2 + \cdots}}.
	$$

\subsection{Quasi-polynomials}

If $f(n)$ is a function of $\FZ$ such that 
$$f(n):= c_d(n)n^d +c_{d-1}(n) n^{d-1} +\ldots+c_0(n)$$ 
for some periodic functions $c_k(n)$  then we call $f(n)$  a \textit{quasi-polynomial} of degree $d$.
When $c_k(n)$ has a common period $q$ for $k=0,1,\cdots,d$, %then 
we say $f(n)$ has (quasi-)period $q$. 
We call $c_k(n)$ the $k$-th coefficient (function) of $f(n)$.

\begin{prop}\label{quasi}
$f(n)=c_d(n)n^d+c_{d-1}(n)n^{d-1}+\cdots+c_0(n)$ is quasi-polynomial of period $q$
if and only if for $n=qk+r$ $$f(n)=a_d(r)k^d+a_{d-1}(r)k^{d-1}+\cdots+a_0(r).$$
Moreover, for $j=0,1,\cdots ,d$,  $$c_j(r)=\sum_{i=j}^{d}a_i(r)
\begin{pmatrix}
      i   \\
      j 
\end{pmatrix} (-r)^{i-j}q^{-i}.$$
\end{prop}

\begin{proof}
Substitute $k$ with $q^{-i}(n-r)$ in the first expression of $f(n)$. Rearranging this as a form of polynomial in $n$, we obtain the second expression. $c_j(r)$ is easily obtain from the rearrangement. 
\end{proof}

\subsection{Main theorem}

Our main theorem is as follows:

\begin{thm}\label{main}
Let $\{K_n=\FQ(\sqrt{f(n)})\}_{n\in \FN}$ be a family of real quadratic fields where $f(n)$ is a positive square free integer for each $n$.
Suppose $\fb_n$ is an integral ideal relatively prime to $q$ such
that $\fb_n^{-1} = [1,\delta(n)]$. 
Assume %the continued
%fraction expansion of  
$\delta(n)-1$  has purely periodic continued fraction expansion% is of the following form:
$$\delta(n)-1=[[a_0(n),a_1(n),\cdots,a_{s-1}(n)]]$$
of % purely 
period $s$ independent of $s$ and $a_i(x)$ are polynomials of degree smaller than $d$. %  of a fixed length $s$
%has a fixed length $s$ 
%independent of $n$ and  $a_i(x)\in \FZ[x]$. 
%Let $$d:=\max_{0\leq i\leq s-1 }\{deg(a_i(x))\}.$$
If $N_{K_n}(\fb_n(C+D\delta(n)))$ modulo $q$  is a function only depending on $C$, $D$ and $r$ for $n=qk+r$, %$r=0,1,\ldots,q-1$ (ie. determined by the residue of $n$ by $q$), %
%then
then 
$$
\zeta_q(0,\fb_n(C+D\delta(n))) = A_0(r) + A_1(r) n + \ldots + A_d(r)n^d
$$ 
%is quasi-polynomial in the variable $n$ of degree $d$ and period $q$.
where $n=qk+r, 0\le r< q$ and for some constants $A_0(r), A_1(r), \ldots, A_d(r)$ depending on $r$.
Furthermore, $A_i(r)\in \frac{1}{12 q^{i+2}}\FZ$.
%Moreover, the values of $i$-th coefficient functions are in $\frac{q^{i-2}}{12}\FZ$.
\end{thm}

\begin{cor} 
\label{2.5}Let $\chi$ be a Dirichlet character modulo $q$ for a positive integer $q$ and $\chi_n$ be a ray class character modulo $q$ defined by $\chi\circ N_{K_n}$. With the same assumption of Theorem \ref{main}, we have  
$L_{K_n}(0,\chi_n,\fb_n)$ is quasi-polynomial with degree $d$ and period $q$ and the values of $i$-th coefficient functions are in $\frac{1}{12q^{i+2}}\FZ[\chi(1),\chi(2)\cdots\chi(q)].$
\end{cor}

%\begin{rem}
Note in \cite{J-L}, the linearity is the quasi-linearity in the second form in Proposition \ref{quasi} written as a polynomial in $k$, while we take 
the first form as shape of polynomial in $n$ in this article.
%\end{rem}

\section{Shintani-Zagier decomposition and partial zeta values}

A real quadratic field $K$ is diagonally embedded into its Minkowski space $K_{\FR} \simeq \FR^2$ by $\iota=(\tau_1,\tau_2)$, where $\tau_1,\tau_2$ are two real embeddings of $K$.
The multiplicative action of ${E_q}^+$ on $K^+$ induces an action on $(\FR^2)^+$ by extending in
coordinate-wise way:
% multiplication: 
%There are natural action of ${E_q}^+$ on the set $(\FR^2)^+$ by
$$\epsilon\circ(x,y)=(\tau_1(\epsilon)x,\tau_2(\epsilon)y).$$
%It is well known that 
A fundamental domain $\frak{D}_{\FR}$ of $(\FR^2)^+/E_q^+$ %this action is 
%Then it is easily proved that 
is given as
\begin{equation}\label{fun}
\frak{D}_{\FR} := %(\FR^2)^+/E_q^+=
\{x\iota(1)+y\iota(\epsilon^{-\lambda})|x>0,y\geq0\} \subset (\FR^2)^+
\end{equation}
where $\epsilon^\lambda$ is the totally positive generator of $E_q$ and $\lambda=[E^+:E^+_q]$. 
%$E_q^+=\left<\epsilon^{\lambda}\right>$ for an integer $\lambda$ and $\epsilon>1$ is the unique totally positive fundamental unit. %totally positive unit. 

We fix an integral ideal $\fb$ 
% be an integral ideal of $K$ 
such that $\fb^{-1}=[1,\delta]$. Moreover we assume that $\delta>1$ and $0 <\delta'<1$.
If we take the convex hull of  $\iota(\fb^{-1})\cap(\FR^2)^+$ in $(\FR^2)^+$, the lattice points
on the boundary are  $\{P_i\}_{i\in\FZ}$ for $P_i \in \iota(\fb^{-1})$.
% and 
$P_i$ are uniquely determined by imposing %the inequalities 
that
%Let $\{P_i\}_{i\in\FZ}$ be the set of boundary points of the convex hull of $\iota(\fb^{-1})\cap(\FR^2)^+$
$P_0=\iota(1), P_{-1}=\iota(\delta)$ and $x(P_i)<x(P_{i-1})$ where $x(P_k)$ is the first coordinate of $P_k$. % for 
%$k\in\FZ$. 
Since $$P_{\lambda m}=\iota(\epsilon^{-\lambda})$$ for some positive integer $m$ (See Proposition 2.4 (5) in \cite{J-L} ), $\frak{D}_\FR$  %the fundamental domain $(\FR^2)^+/E_q^+$ 
is further decomposed into % expressed as 
$(\lambda\cdot m)$-disjoint union of smaller %`half closed' 
cones:
$$
\frak{D}_\FR = %(\FR^2)^+/E_q^+=
\bigsqcup_{i=1}^{\lambda m}\{xP_{i-1}+yP_i\,\,|\,\,x>0,\,\,y\geq0\}.
$$

Accordingly the fundamental set of the quotient
$(\iota(\frac{C+D\delta}{q}+\fb^{-1}) \bigcap {(\FR^2)}^+)/E_q^+$ inside $\frak{D}_\FR$, which
we denote by $\frak{D}$
is given by a disjoint union:
%$$\frac{C+D\delta}{q}+\fb^{-1} \bigcap \FR_2^+/E_q(K)^+=
$$
\frak{D} := \bigsqcup_{i=1}^{\lambda m}
\Big{(}\iota(\frac{C+D\delta}{q}+\fb^{-1})\bigcap\{xP_{i-1}+yP_i\,\,|\,\,x>0,\,\,y\geq0\} \Big{)}.
$$

Since $\{P_{i-1},P_i\}$ is a $\FZ$-basis of $\iota(\fb^{-1})$,
there is a unique $(x_{C+D\delta}^i,y_{C+D\delta}^i)\in (0,1]\times[0,1)$
such that
\begin{equation}\label{xy}
%0<x_{C+D\delta}^i\leq1,\,\,0\leq y_{C+D\delta}^i<1,\,\,
x_{C+D\delta}^iP_{i-1}+y_{C+D\delta}^iP_i\in\iota(\frac{C+D\delta}{q}+\fb^{-1}),
\end{equation}
for each $i,C,D\in\FZ.$
Thus
\begin{equation}\label{set}
\begin{split}
\iota\big(\frac{C+D\delta}{q}+\fb^{-1}\big)\bigcap\{xP_{i-1}+yP_i\,\,|\,\,x>0,\,\,y\geq0\} \\
=\{(x_{C+D\delta}^i+n_1)P_{i-1}+(y_{C+D\delta}^i+n_2)P_i\,\,|\,\, n_1, n_2 \in\FZ_{\geq0}\}.
\end{split}
\end{equation}

In \cite{Yamamoto}, Yamamoto found a recursive relation satisfied by $(x_{C+D\delta}^i,y_{C+D\delta}^i)$: % relations:
\begin{equation}\label{Yam}
\begin{split}
x_{C+D\delta}^{i+1} & =\langle  b_ix_{C+D\delta}^i+y_{C+D\delta}^i\rangle,\\
y_{C+D\delta}^{i+1} & =1-x_{C+D\delta}^i,
\end{split}
\end{equation}
where $\langle\cdot\rangle$ is as defined %at the end of the introduction. 
%(ie.
as 
$
\langle x \rangle=
 x-[x]$ (resp. $1$) for $x\not\in \FZ$ (resp. for $x\in \FZ$)(See (2.1.3) of {\it loc.sit.}).

Let  $A_i:=x(P_i)$ and $\delta_i=\frac{A_{i-1}}{A_i}$ for all $i\in\FZ$. Then from Eq.(\ref{set}), we obtain the following: 
\begin{equation*}\label{e1}
\zeta_q(s,(C+D\delta)\fb)
=\sum_{i=1}^{\lambda m}\sum_{n_1,n_2\geq0}N((x_{C+D\delta}^i+n_1)\delta_i+(y_{C+D\delta}^i+n_2))^{-s} A_i^{-s}.
\end{equation*}

From Shintani and Yamamoto's evaluation of zeta function at $s=0$ (Lemma 2.5-6 in\cite{J-L}), we find an expression of the partial zeta value in terms of the values of 1st and 2nd Bernoulli polynomials: 
\begin{equation}\label{eq}
\zeta_q(0,(C+D\delta)\fb)
=\sum_{i=1}^{\lambda m} -B_1(x_{C+D\delta}^{i})B_1(x_{C+D\delta}^{i-1})+\frac{b_i}{2}B_2(x_{C+D\delta}^{i})
\end{equation}

Moreover we can express $x_{C+D\delta}^{mi+j},y_{C+D\delta}^{mi+j}$ using epsilon action $\ast$ defined in (\ref{action}).
\begin{lem}\label{epsilon}
Let $\epsilon$ be the totally positive fundamental  unit of $K$. Then  for $i>1$ we have
$$
x_{C+D\delta}^{mi+j}=x_{\epsilon^{i}\ast(C+D\delta)}^j \quad\text{and}\quad y_{C+D\delta}^{mi+j}=y_{\epsilon^{i}\ast(C+D\delta)}^j,
$$
for $j=0,1,2,\cdots,m-1$.
\end{lem}
\begin{proof}
See Lemma 2.7 in \cite{J-L}.
\end{proof}
From above lemma and equation (\ref{eq}), we obtain the following Lemma

\begin{lem}\label{sum}
\begin{equation*}
\begin{split}
&\zeta_q(0,(C+D\delta)\fb)
\\
=&\sum_{i=1}^{m}\sum_{j=0}^{\lambda -1} -B_1(x_{\epsilon^{j}\ast(C+D\delta)}^i)B_1(y_{\epsilon^{j}\ast(C+D\delta)}^i)+
\frac{b_i}{2}B_2(x_{\epsilon^{j}\ast(C+D\delta)}^i).
\end{split}
\end{equation*}
\end{lem}

Let $F_{\delta}$ be a set defind in Eq.(\ref{Funda}).

 \begin{lem}\label{CD-orbit} If $\fb$ is an integral ideal such that $\fb^{-1}=[1,\delta]$ with $(\fb,q)=1$. Then for $C+D\delta\in F_{\delta},$ we find that 
$$Orb(C+D\delta)=\{\epsilon^j\ast(C+D\delta) \,\,|\,\,j=0,1,\cdots \lambda-1 \},$$
where $\epsilon$ is a totally positive fundamental unit of $E^+$ and $\lambda=[E^+:E^+_q]$
\end{lem}
\begin{proof}
By  definition of  $\epsilon\ast$ in (\ref{action}) we find that 
$\epsilon^j \ast (C+D\delta)=C+D\delta$ if and only if 
\begin{equation}\label{bac}
 (\epsilon^j-1)(C+D\delta)\in q \fb^{-1}.
\end{equation} 
Since $(q, \fb(C+D\delta))=1$, (\ref{bac}) is equal to 
$$\epsilon^j\in E_q^+.$$

\end{proof}

Combining Lemma \ref{sum} and  Lemma \ref{CD-orbit},  we obtain 
\begin{prop}\label{orbit}
$$\zeta_q(0,(C+D\delta)\fb)=\sum_{(A,B)\in Orb(C+D\delta)}\sum_{i=1}^{m}-B_1(x_{A+B\delta}^{i})B_1(x_{A+B\delta}^{i-1})+\frac{b_i}{2}B_2(x_{A+B\delta}^{i}).$$
\end{prop}

\section{Periodicity of orbit}

In this section, we adapt the discussion of the previous section to  a family of real quadratic fields. With the assumptions of this paper, in the considered family, 
the Shintani-Zagier cone decomposition 
varies. This variation of decomposition  is far from being periodic in $q$ but the fundamental unit acts on $F_\delta$ with  period $q$. 

We restate the assumption for the family $(K_n, \fb_n)$. 
Let $\{K_n=\FQ(\sqrt{f(n)})\}_{n\in \FN}$ be a family of real quadratic fields where $f(n)$ is a positive square free integer for each $n$. Suppose $\fb_n$ is an integral ideal relatively prime to $q$ such
that $\fb_n^{-1} = [1,\delta(n)]$ for a $\delta(n)$ satisfying
$\delta(n) > 2$, $0<\delta(n)'<1$. 
Assume the continued fraction expansion of  $\delta(n)-1$ is 
$$
\delta(n)-1=[[a_0(n),\cdots,a_{s-1}(n)]].
$$
Then it is known  that the minus continued fraction expansion of $\delta(n)$ is %has minus continued fraction expression
%
%Then the minus continued fraction expansion of $\delta(n)$ is obtained from the above continued
%fraction expansion of $\delta(n)-1$ as follows:
$$
\delta(n)=((b_0(n),b_1(n),\cdots,b_{m(n)-1}(n)))
$$ 
where 
$$b_i(n)=\begin{cases}a_{2j}(n)+2 & \text{ for } i=S_j(n)\\2 &\text{ otherwise}\end{cases}$$
and 
for some index $S_j(n)$ depending on $n$. 
$S_j(n)$ is defined from  $a_i(n)$ as follows:
$$S_j(n)=\begin{cases}0 & \text{ for } j=0\\S_{j-1}(n)+a_{2j-1}(n) &\text{ for } j\geq1\end{cases}$$
It follows that the period $m(n)$  of the minus continued fraction of $\delta(n)$ is 
\begin{equation}\label{period}
m(n)=\begin{cases}S_{\frac{s}{2}}(n) & \text{ for even } s\\S_{s}(n) &\text{ for odd } s\end{cases}
\end{equation}
(cf. page 177-178 in \cite{Zagier}, Lemma 3.1 in \cite{J-L}).

Setting $\mu(s)= 1$ (resp. $2$) for even $s$(resp. odd $s$),
%\begin{equation*}
%\mu(s)=\begin{cases}\frac{1}{2} & \text{ for even } s\\1 &\text{ for odd } s\end%{cases}
%\end{equation*}
we have 
$$m(n)=S_{\mu(s)s}(n).$$

Let $\epsilon_n>1$ be the totally positive fundamental unit of $K_n$ and 
$\epsilon_n^{\lambda(n)}>1$ be the generator of $E_q(K_n)^+$. % with $\lambda(n)\in \FZ^+$.
Then $\lambda(n)=[E^+(K_n):E^+_q(K_n)]$ equals the cardinality of the orbit of $(C,D)\in F_{\delta(n)}$ as we have discussed in the previous section.

For  each $(C,D)\in F_{\delta(n)}$ pair, we have the sequence $\{x_{C+D\delta(n)}^i,y_{C+D\delta(n)}^i\}_{i\geq-1}$ defined as in Eq.(\ref{xy}) of  the previous section: %satisfying
\begin{eqnarray}
x_ {C+D\delta(n)}^0=& \left<\frac{D}{q}\right>,\\
y_ {C+D\delta(n)}^{0}=&\frac{C}{q},\\
%\end{eqnarray}
%
%\begin{eqnarray}
x_ {C+D\delta(n)}^{i+1}=&\left<b_i(n)x_ {C+D\delta(n)}^i-x_{C+D\delta(n)}^{i-1} \right>,\\
y_ {C+D\delta(n)}^{i+1}=&1-x_ {C+D\delta(n)}^{i}.
\end{eqnarray}

%Soon we will see that the sequence $\{x_{C+D\delta(n)}^i,y_{C+D\delta(n)}^i\}_{i\geq-1}$
%has actually $q$-periodicity in $n$. 

%Let $\epsilon_n>1$ be the fundamental unit of $K_n.$ 
For  $i\geq 0$, we define
$0\leq C_i(n), D_i(n) \leq q-1$ as follows
$$\epsilon_n^{i}\ast(C+D\delta(n))=C_i(n)+D_i(n),$$
with $C_0(n)=C, D_0(n)=D.$

Using Lemma \ref{epsilon} , we have
$$x_{C+D\delta(n)}^{m(n)i}=x_{\epsilon_n^i\ast(C+D\delta(n))}^0=x_{C_i(n)+D_i(n)\delta(n)}^0=\left<\frac{D_i(n)}{q}\right>$$
and
$$y_{C+D\delta(n)}^{m(n)i}=1-x_{C+D\delta(n)}^{m(n)i-1}=y_{\epsilon_n^i\ast(C+D\delta(n))}^0=y_{C_i(n)+D_i(n)\delta(n)}^0=\frac{C_i(n)}{q}.$$
%We note that 
%$$ m(n)=S_{\mu(s)s}(n).$$

\begin{lem}\label{mul}
For $i\geq1$,
$$m(n)i=S_{\mu(s)si}(n).$$
\end{lem}
\begin{proof}
We can rewrite $S_{\mu(s)si}(n)$ as following:
$$ \sum_{j=1}^{i}\sum_{k=1}^{s\mu(s)}a_{2(j-1)\mu(s)s+2k-1}(n)= \sum_{j=1}^{i}\sum_{k=1}^{s\mu(s)}a_{2k-1}(n)= i S_{\mu(s)s}(n)=i m(n).$$
\end{proof}

Finally we obtain
$$\left<\frac{D_i(n)}{q}\right>=x_{C+D\delta(n)}^{S_{\mu(s)si}(n)}$$
and %\,\,\,\,\,\,
$$\frac{C_i(n)}{q}= 1-x_{C+D\delta(n)}^{S_{\mu(s)si}(n)-1}.$$

Note that for $j>0$, 
$x_{C+D\delta(n)}^{S_{j}(n)}(n)$ and
$ x_{C+D\delta(n)}^{S_{j}(n)-1}$ are determined only by $r$ for $n=qk+r$(See Property 3 in p.16 and Proposition 3.5 of \cite{J-L}).
Thus we have $q$-invariant property of $C_i(n), D_i(n)$  in $n=qk+r$:
\begin{prop}\label{CD}
For $n,n'$ such that $n'=n+qk$, assume $(C,D)\in F_{\delta(n)}\cap F_{\delta(n')}$.
Then $(C_i(n),D_i(n)) = (C_i(n'),D_i(n') )\in F_{\delta(n)}\cap F_{\delta(n')}$.
%$C_i(n)=C_i(n')$ and  $D_i(n)=D_i(n')$. % are invariant as $k$ varies.
\end{prop}

Consequently, we see
$\lambda(n) = \lambda(n')$ and 
$E^+(n)/E^+_q(n) \simeq E^+(n')/E^+_q(n')$  for $n'= n+qk$. Furthermore,
from the original assumption that $N((C+D\delta(n))\pmod{q}$ is determined by
$r$ the residue of $n$, we have
$F_{\delta(n)} = F_{\delta(n')}$.
This  identification preserves the $E^+(n)/E^+_q(n)$-action. % in this case.

%\begin{proof}

%Thus  from Proposition \ref{inv} and Lemma \ref{mul} we find that for $n=qk+r$

%and for positive integer $t$, we have 
%\begin{equation}\label{Si}
%\begin{split}
%&x_{C+D\delta(n)}^{S_{t}(n)-1}=x_{C+D\delta(n)}^{S_{t-1}(n)+a_{2t-1}(n)-1}=x_{C+D\delta(n)}^{S_{t-1}(n)+\gamma_{2t-1}(r)-1}\\
%&=\nu_{CD}^{\Gamma_{t-1}(r)+\gamma_{2t-1}(r)-1}(r)=\nu_{CD}^{\Gamma_{t}(r)-1}(r)
%\end{split}
%\end{equation}
%From equation (\ref{Si}) , we have 
%\begin{equation*}
%\begin{split}
%& y_{C+D\delta(n)}^{m(n)i}=1-x_{C+D\delta(n)}^{m(n)i-1}\\
%&= 1-x_{C+D\delta(n)}^{S_{\mu(s)si}(n)-1}=1-\nu_{CD}^{\Gamma_{\nu(s)si}(r)-1}(r).
%\end{split}
%\end{equation*}
%From this we complete the proof.

%\end{proof}
%We recall that
%$$F_{\delta}=\{(C,D)\in \FZ^2 | \,\, 0\leq C,D\leq q-1, ((C+D\delta)\fb,q)=1\}.$$

\begin{lem}
If $N((C+D\delta(n))\fb_n) \pmod{q}$ is determined by $r$ the residue of $n$ mod $q$ 
for $n=qk+r$ 
then $F_{\delta(n)}$ is invariant as $k$ varies.
\end{lem}
\begin{proof}
%Since $(C,D)\in F_{\delta(n)}$ 
Use the fact $((C+D\delta(n))\fb_n,q)=1$ iff $(N((C+D\delta(n)\fb_n),q)=1$.
This is again equivalent to %We note that 
$(C,D)\in F_{\delta(n)}$. 
From the $q$-invariance of % is equivalent to 
$(N((C+D\delta(n))\fb_n),q)=1$ in $n$, 
%Then the assumption yields 
we obtain that $(C,D)\in F_{\delta(n)}$ iff $(C,D)\in F_{\delta(n')}$ for $n'=n+qk$.
%invariance property.
%Thus by our assumption we complete the
%proof. 
\end{proof}

%\begin{lem}\label{Lambda}
%If $N((C+D\delta(n))\fb_n) \pmod{q}$ is a function only depending on $C,D$ and $r$ for $n=qk+r$ then $\lambda(n)$ is invariant as $k$ varies.
%\end{lem}

%\begin{proof}

%We suppose that $n_1\equiv n_2\pmod{q}$ and $K_{n_1}$, $K_{n_2}$ are defined. 
%We note that for $(C,D)\in F_{\delta(n)}$ (See Lemma \ref{CD-orbit}), 
%\begin{equation}\label{action}
%\begin{split}
%&\epsilon_n^j\ast(C+D\delta(n))=C+D\delta(n)\\
%&\leftrightarrow \lambda(n) | j.
%\end{split}
%\end{equation}
%Thus for $(C,D)\in F_{\delta(n_1)}$, we have 
%$$\epsilon_{n_1}^{\lambda(n_1)}\ast(C+D\delta(n_1))=C_{\lambda(n_1)}(n_1)+D_{\lambda(n_1)}(n_1)\delta(n_1)=C+D\delta(n_1).$$
%From Lemma \ref{CD}, we have 
%$$C_{\lambda(n_1)}(n_1)=C_{\lambda(n_1)}(n_2)=C$$
%and
%$$D_{\lambda(n_1)}(n_1)=D_{\lambda(n_1)}(n_2)=D.$$
%Thus 
%$$ \epsilon_{n_2}^{\lambda(n_1)}\ast(C+D\delta(n_2))=C+D\delta(n_2).$$
%Hence from (\ref{action}) we have 
%$$\lambda(n_2)|\lambda(n_1).$$
%Conversely we have also 
%$$\lambda(n_1)|\lambda(n_2)$$
%and complete the proof. 
%\end{proof}

%From these properties, we conclude that 
In particular, we obtain the $q$-invariance in $n$ of an orbit in $F_{\delta(n)}$ as in the 
following:
\begin{prop}\label{period-orbit}If $N((C+D\delta(n))\fb_n) \pmod{q}$ is a function  depending only 
on $C,D$ and $r$ for $n=qk+r$ then
$Orb(C+D\delta(n))$ is invariant as $k$ varies.
\end{prop}
%\begin{proof}
%From Lemma \ref{CD-orbit}, we have
%$$Orb(C+D\delta(n))=\{(C_i(n),D_i(n)) \,\,|\,\, j= 0,1,2,\cdots \lambda(n)-1\}.$$
%We suppose that $n_1\equiv n_2\pmod{q}$ and $K_{n_1}$, $K_{n_2}$ are defined. 
%From Proposition \ref{CD} and Lemma \ref{Lambda}, we find that
%$$Orb(C+D\delta(n_1))=Orb(C+D\delta(n_2)).$$
%\end{proof}

\section{Explicit computation of the coefficients}
Recall some calculations from Section 3.2 in \cite{J-L}.
%To see periodicity of orbit, we review the periodicity of  the sequence $\{x_{C+D\delta(n)}^i,y_{C+D\delta(n)}^i\}_{i\geq-1}$ that was observed in Section 3.2 in \cite{J-L}.
For this we define a sequence $\{\nu_{CD}^i(r)\}_{i\geq -1}$ for $0\leq r\leq q-1$. As we have constrained that $a_i(x)\in \FZ[x]$,
 $\langle a_i(n)\rangle_q$ is independent of $k$ for $n=qk+r$ but determined only by $r$.  
 %as $k$ varies for $n=qk+r$.
Thus we can define
$$\gamma_{i}(r):=\langle a_i(n)\rangle_q= \langle a_i(r) \rangle_q.$$

Let
$$\Gamma_j(r):=
\begin{cases}\Gamma_{j-1}(r)+\gamma_{2j-1}(r) &\text{for $ j\geq 1$}\\ 0 & \text{for $ j=0$}\end{cases}$$
For $i\geq 1$,
$$
c_i(r) = \begin{cases}
\gamma_{2j}(r) + 2  & \text{for $i=\Gamma_j(r)$} \\
2 & \text{otherwise}
\end{cases}
$$

Now we can define a sequence $\{\nu_{CD}^i(r)\}_{i\geq -1}$ satisfying% follows:
\begin{equation*}
\nu_{CD}^{-1}(r)=\frac{q-C}{q},\quad \nu_{CD}^{0}(r)= \langle \frac{D}{q}\rangle
\end{equation*}
and
$$\nu_{CD}^{i+1}(r)=\langle c_i(r)\nu_{CD}^{i}(r)-\nu_{CD}^{i-1}(r)\rangle. $$

When  $C,D$ are fixed and  clear from the context,  we denote $x_{C+D\delta(n)}^{i}$ and $\nu_{CD}^{i}(r)$ by $x_i(n)$ and  $\nu_i(r)$,  respectively.

\begin{prop}\label{pe}For $ j\geq 0$ and integer $n$ such that $a_{2j+1}(n)\geq q$,
$\{x_i(n)\}_{S_j(n)\leq i\leq S_{j+1}(n)}$ has period $q$. Explicitly we have
$$ x_{S_j(n)+q+i}(n)=x_{S_j(n)+i}(n)\,\,\text{for}\,\,
0\leq i \leq a_{2j+1}(n)-q.$$
\end{prop}
\begin{proof}
See Proposition 3.4 in \cite{J-L}
\end{proof}

\begin{prop}\label{inv}
For integers $j\geq 0$, if $n=qk+r$ then
$$x_{S_j(n)+i}(n)=\nu_{\Gamma_j(r)+i}(r)\,\,\text{for}\,\,0\leq i\leq \gamma_{2j+1}(r).$$
\end{prop}
\begin{proof}
See Proposition 3.5 in \cite{J-L}
\end{proof}

For $0\leq i\leq s-1$, let $a_i(x)=\sum_{j=0}^{d}\alpha_{ij}x^j\in \FZ[x]$. Then for $0\leq r\leq q-1$, there are an unique $\gamma_i(r)\in[1,q]\cap\FZ$ such that  $a_i(r)=q\tau_i(r)+\gamma_i(r)$ for
$\tau_i(r)\in\FZ$.
\begin{lem}\label{degree}
If $n=qk+r$, then 
we have
$$a_i(n)=q\sum_{m=1}^{d}A_{im}(r)k^m+q\tau_i(r)+\gamma_i(r),$$
where $A_{im}(r)=\sum_{j=m}^{d}\alpha_{ij} \begin{pmatrix}
     j    \\
     m
\end{pmatrix} q^{m-1} r^{j-m}.$
\end{lem}

We recall that $\{x_{C+D\delta(n)}^i\}_{S_j(n)\leq i \leq S_{j+1}(n)}$ is and arithmetic progression mod $\FZ$ with difference $\langle x_{C+D\delta(n)}^{S_j(n)+1}-x_{C+D\delta(n)}^{S_j(n)}\rangle$[See Proposition 3.2 in \cite{J-L}]. From proposition \ref{inv}, we find that  if $n=qk+r$ then 
\begin{equation}\label{diff}
\langle x_{C+D\delta(n)}^{S_j(n)+1}-x_{C+D\delta(n)}^{S_j(n)}\rangle=\langle \nu_{CD}^{\Gamma_j(r)+1}(r)-\nu_{CD}^{\Gamma_{j}(r)}(r)\rangle.
\end{equation}
Let $d_{CD}^{j}(r):=\langle \nu_{CD}^{\Gamma_{j}(r)+1}(r)-\nu_{CD}^{\Gamma_{j}(r)}(r)\rangle$. 
%and for fixed $C, D$ we denote $d_{CD}^{j}(r)$ by $d_j(r)$.

Now we will express the value $\zeta(0,(C+D\delta(n))\fb(n))$ using $q$ $k$ $r$ where $n=qk+r$.
We note that form Proposition \ref{orbit}
\begin{equation} \label{zeta-value}
\begin{split}
&\zeta_q(0,(C+D\delta(n))\fb(n))\\
&=\sum_{(A,B)\in Orb(C+D\delta(n))}\sum_{i=1}^{m(n)}-B_1(x_{A+B\delta(n)}^{i})B_1(x_{A+B\delta(n)}^{i-1})+\frac{b_i(n)}{2}B_2(x_{A+B\delta(n)}^{i}).
\end{split}
\end{equation}

For simplification, we define $$F(x,y):=-B_1(x)B_1(y)+B_2(x)=(x-\frac{1}{2})(\frac{1}{2}-y)+x^2-x+\frac{1}{6}.$$

Because  $b_i(n)=2$ if  $i \ne S_j(n)$  for some $j$, 
we can find that 
\begin{equation}\label{divide}
\begin{split}
&\sum_{i=1}^{m(n)}-\big(B_1(x_{C+D\delta(n)}^{i})B_1(y^i_{C+D\delta(n)})+\frac{b_i(n)}2B_2(x_{C+D\delta(n)}^{i})\big)\\
&=\sum_{l=1}^{s\mu(s)}\big(-B_1(x^{S_l(n)}_{C+D\delta(n)})B_1(x^{S_l(n)-1}_{C+D\delta(n)})+\frac{a_{2l}(n)+2}{2}B_2(x^{S_l(n)}_{C+D\delta(n)})\big)\\
&+\sum_{l=0}^{s\mu(s)-1}\sum_{i=S_l(n)+1}^{\substack{S_{l+1}(n)-1}}F(x^i_{C+D\delta(n)},x^{i-1}_{C+D\delta(n)})
\end{split}
\end{equation}

From the fact that  $\{x_{C+D\delta(n)}^i\}_{S_j(n)\leq i \leq S_{j+1}(n)}$ is and arithmetic progression mod $\FZ$ with difference $d_{CD}^j(r)$ for $n=qk+r$, we obtain the following:
If $1\leq\gamma\leq q$ and  $a_{2l+1}(n)\geq \gamma$ then 
\begin{equation}\label{sum1}
\begin{split}
&\sum_{i=S_{l}(n)+1}^{S_{l}(n)+\gamma}F(x_i(n),x_{i-1}(n)) \\
&=\frac{1}{12}\Big{(}6(\gamma d_{l}(r)^2
+(1-2d_{l}(r))[\nu_{\Gamma_{l}(r)}(r)+d_{l}(r)\gamma]_1+B_2(x_{S_{l}(n)+\gamma}(n))-B_2(x_{S_{l}(n)}(n)))-\gamma \Big{)}
\end{split}
\end{equation}
Moreover if $\gamma=q$ then  from the periodicity of $x_i(n)$ [See Proposition \ref{pe}],
we have the periodicity of the values of the 2nd Bernoulli polynomial: 
$$B_2(x_{S_{l}(n)+q}(n))=B_2(x_{S_{l}(n)}(n)). $$
Thus  
\begin{equation}\label{sum2}
\sum_{i=S_{l}(n)+1}^{S_{l}(n)+q}F(x_i(n),x_{i-1}(n))
=\frac{1}{12}\Big{(}6(qd_{l}(r)^2
+(1-2d_{l}(r))[\nu_{\Gamma_{l}(r)}(r)+d_{l}(r)q]_1)-q \Big{)}.
\end{equation}
%and 
%where $B_2(x)$ is the second Bernoulli polynomial.

We note that $Orb(C+D\delta(n))$ is a set depending only $C,D$ and $r$ for $n=qk+r$ [See Proposition \ref{period-orbit}] under the conditions of the following proposition. Thus 
for $n=qk+r$ we can define $$Orb(C+D\delta(n))=:Orb_{CD}(r).$$
\begin{prop}\label{summ}
If $\fb_n^{-1}=[1,\delta(n)]$ and 
$$ \delta(n)-1=[[a_0(n),a_1(n), \cdots a_{s-1}(n)]]$$
for $a_i(x)=\sum_{j=0}^d\alpha_{ij}x^{j}\in\FZ[x]$ and if 
$N((C+D\delta(n))\fb_n) \pmod{q}$ is a function only depending on $C,D$ and $r$
then
we have for $n=qk+r$
$$\zeta_q((C+D\delta(n))\fb_n)=\sum_{(A,B)\in Orb_{CD}(r)}B_{AB}^0(r)+B_{AB}^1(r)k+\cdots B_{AB}^d(r)k^d,$$
where for $m\geq 1,$
\begin{equation*}
\begin{split}
&B_{AB}^m(r)=\frac{q}{2}\sum_{l=1}^{s\mu(s)}A_{2l,m}(r)B_2(\nu_{AB}^{\Gamma_l(r)}(r))\\
&+\frac{1}{12}\sum_{l=0}^{s\mu(s)-1}A_{2l+1,m}(r)\Big{(}6(q{d_{AB}^l(r)}^2+(1-2{d_{AB}^l(r)})[\nu_{AB}^{\Gamma_l(r)}(r)+d_{AB}^l(r)q]_1)-q\Big{)}
\end{split}
\end{equation*}
and 
\begin{equation*}
\begin{split}
&B_{AB}^0(r)=\sum_{l=1}^{s\mu(s)}-B_1(\nu_{AB}^{\Gamma_l(r)}(r))B_1(\nu_{AB}^{\Gamma_l(r)-1}(r))+\frac{q\tau_{2l}(r)+\gamma_{2l}(r)+2}{2}B_2(\nu_{AB}^{\Gamma_l(r)}(r))\\
&+\frac{1}{12}\sum_{l=0}^{s\mu(s)-1}\Big{[}\tau_{2l+1}(r)\Big{(}6(q{d_{AB}^l(r)}^2+(1-2{d_{AB}^l(r)
})[\nu_{AB}^{\Gamma_l(r)}(r)+d_{AB}^l(r) q]_1)-q\Big{)}\\
&+6B_2(\nu_{AB}^{\Gamma_{l+1}(r)-1}(r))-6B_2(\nu_{AB}^{\Gamma_{l}(r)}(r))+\Big{(}(\gamma_{2l+1}(r)-1){d^{l}_{AB}(r)}^2\\
&+(1-2d^{l}_{AB}(r))[\nu^{\Gamma_{l}(r)}_{AB}(r)+d^{l}_{AB}(r)(\gamma_{2l+1}(r)-1)]_1 \Big{)}-\gamma_{2l+1}(r)+1\Big{]}.
\end{split}
\end{equation*}
\end{prop}

\begin{proof}
%From Lemma \ref{degree}, we have
%\begin{equation*}\label{su2}
%a_{2l}(n)=q\sum_{m=1}^{N}A_{2l,m}(r)k^m+q\tau_{2l}(r)+\gamma_{2l}(r).
%\end{equation*}
%and
%\begin{equation*}\label{su3}
%S_{l+1}(n)=S_{l}(n)+a_{2l+1}(n)=S_l(n)+q\sum_{m=1}^{N}A_{2l+1,m}(r)k^m+q\tau_{2l+1}(r)+\gamma_{2l+1}(r).
%\end{equation*}
Thus from equation (\ref{divide}) and  lemma \ref{degree}, we obtain the following:
\begin{equation}\label{divide2}
\begin{split}
\sum_{i=1}^{m(n)} &\big(B_1(x_{A+B\delta(n)}^{i})B_1(y^i_{A+B\delta(n)})+\frac{b_i(n)}2B_2(x_{A+B\delta(n)}^{i})\big)=\\
&\sum_{l=1}^{s\mu(s)}[-B_1(x^{S_l(n)}_{A+B\delta(n)})B_1(x^{S_l(n)-1}_{A+B\delta(n)})\\
&+
\frac{q\sum_{m=1}^{d}A_{2l,m}(r)k^m+q\tau_{2l}(r)+\gamma_{2l}(r)+2}{2}B_2(x^{S_l(n)}_{A+B\delta(n)})]\\
&+\sum_{l=0}^{s\mu(s)-1}\sum_{i=S_l(n)+1}^{\substack{S_l(n)+q\sum_{m=1}^{d}A_{2l+1,m}(r)k^m\\+q\tau_{2l+1}(r)+\gamma_{2l+1}(r)-1}}F(x^i_{A+B\delta(n)},x^{i-1}_{A+B\delta(n)})
\end{split}
\end{equation}
Since $\{F(x^i_{A+B\delta(n)},x^{i-1}_{A+B\delta(n)})\}_{S_l(n)+1\leq i\leq S_{l+1}(n)-1}$ has period $q$, we have the following:

\begin{equation}\label{su4}
\begin{split}
&\sum_{i=S_l(n)+1}^{\substack{S_l(n)+q\sum_{m=1}^{d}A_{2l+1,m}(r)k^m\\+q\tau_{2l+1}(r)+\gamma_{2l+1}(r)-1}}  F(x^i_{A+B\delta(n)},x^{i-1}_{A+B\delta(n)})\\
=&(\sum_{m=1}^{d}A_{2l+1,m}(r)k^m+\tau_{2l+1}(r)) \sum_{i=S_l(n)+1}^{S_l(n)+q}F(x^i_{A+B\delta(n)},x^{i-1}_{A+B\delta(n)})\\
&+\sum_{i=S_l(n)+1}^{S_l(n)+\gamma_{2l+1}(r)-1}F(x^i_{A+B\delta(n)},x^{i-1}_{A+B\delta(n)})
\end{split}
\end{equation}
We note that
\begin{equation}\label{su5}
x^{S_{l}(n)}_{A+B\delta(n)}=\nu^{\Gamma_l(r)}_{AB}(r)
\end{equation}
\begin{equation}\label{su6}
x^{S_{l}(n)-1}_{A+B\delta(n)}=\nu^{\Gamma_l(r)-1}_{AB}(r)
\end{equation}
\begin{equation}\label{su7}
 x^{S_{l}(n)+\gamma_{2l+1}(r)-1}_{A+B\delta(n)}=\nu^{\Gamma_{l+1}(r)-1}_{AB}(r)
\end{equation}
Form (\ref{sum1}),(\ref{sum2}) and (\ref{su5})-(\ref{su7}), we find that (\ref{su4}) is equal to the following:
\begin{equation}
\begin{split}
&\Big[\sum_{m=1}^{d}A_{2l+1,m}(r)k^m+\tau_{2l+1}(r)\Big{]}\cdot \\
&\Big{[}6(q{d_{AB}^l(r)}^2+(1-2{d_{AB}^l(r)})[\nu_{AB}^{\Gamma_l(r)}(r)+d_{AB}^l(r)q]_1)-q\Big{]}\\&+6B_2(\nu_{AB}^{\Gamma_{l+1}(r)-1}(r))-6B_2(\nu_{AB}^{\Gamma_{l+1}(r)}(r))+6 \Big{(}(\gamma_{2l+1}(r)-1){d^{l}_{AB}(r)}^2\\&+(1-2d^{l}_{AB}(r))[\nu^{\Gamma_{l}(r)}_{AB}(r)+d^{l}_{AB}(r)(\gamma_{2l+1}(r)-1)]_1 \Big{)}-\gamma_{2l+1}(r)+1
\end{split}
\end{equation}
\end{proof}

\noindent{\bf Proof of Theorem 2.4}\\

We note that $\nu^{\Gamma_l(r)}_{AB}(r)$, $\nu^{\Gamma_l(r)-1}_{AB}(r)$ and  $d^{l}_{AB}(r) \in \frac{1}{q}\FZ$. Thus for $i=0,1,\cdots,d,$
$$B_{AB}^i(r)\in \frac{1}{12q^2}\FZ.$$ From Proposition 2.3, we complete the proof of theorem 2.4.

\section{Two examples}
We present here two examples of $(K_n, \fb_n)$ showing polynomial behaviour of 
$\zeta_q(0, \fb_n)$.
%the special values of ray class partial zeta
%functions at $s=0$. For these families, the coefficients of the quasi-polynomials are computed in sequel.
For both examples, we take $\fb_n = (C+D\delta(n))O_{K_n}$, for $C+D\delta(n) \in F_{\delta(n)}$.
% are defined in Section 3. 
%The first example is 
The first one is a family of real quadratic fields already appeared in a literature dealing the associated class number
one problem. The second family is given by a quartic polynomial . 

\subsection{Case 1: $f(n)=n^2+2$}
The following example is one of  so-called Richad-Degert type. The quasi-linearity was studied in \cite{Lee4} to  solve
the class number one problem for the family. 
From the computation of partial zeta values for ray class in this article, 
one can recover exactly the result in \textit{loc.cit.} for  partial Hecke's L-values associated 
to a mod-$q$  Dirichlet character $\chi$. 
%which behaves as a quasi-linear polynomial. 
%\begin{exam}
%Let $K_n$ be 
%For a square free $f(n)=n^2+2$, let $K_n:=
%$\FQ(\sqrt{f(n)})$ where $f(n) = n^2+2$. 

For square free $f(n)=n^2+2$, let $K_n= \FQ(\sqrt{f(n)}$.
We fix $\fb_n =O_{K_n}$ the ring of integers in $K$. 
$O_{K_n}= [1,\delta(n)]$ and
%$O_{K_n}$ is $[1,\delta(n)]$. 
%where
%$\delta(n):=\sqrt{f(n)}+n+1.$
%Then 
%$$\delta(n)-1 = [[2n,n]].$$. % has purely periodic continued fraction $[[2n,n]].$ 
%Moreover the totally positive fundamental unit $\epsilon_n$ of $K_n$ is $n^2+1+n\sqrt{f(n)}$.

\begin{itemize}
\item $\delta(n) = \sqrt {f(n)} + n + 1$.
\item The continued fraction of $\delta(n) -1$ is
$ [[2n,n]]$.
\item The totally positive fundamental unit $\epsilon_n$ is $n^2+ 1 + n\sqrt{f(n)}$.
%\item
\end{itemize}

One can easily check that $N(C+D\delta(n))$ is invariant for $n$ modulo $q$. 
Let $n=qk+r$. 
We can describe  the orbit of $C+D\delta(n)$ by the action of $\epsilon_n$ as follows:
%For $C+D\delta(n)\in F_{\delta(n)}$, the orbit of $\epsilon_n$ action $Orb(C+D\delta(n))$ is 
%computed as follows: 
% by the following $\epsilon_n$ action. For $C+D\delta(n)\in F_{\delta(n)}$, we have

$$\epsilon_n \ast (C_i(n)+D_i(n)\delta(n))=C_{i+1}(n)+D_{i+1}(n) \delta(n),$$
where
\begin{equation*}
\begin{split}
&C_0(n)=C, \,\,\,D_0(n)=D\\
&C_{i+1}(n) =\left<C_i(n)(n^2+1)+D_i(n)(2n^3+n^2+3n+1)\right>_q\\
&D_{i+1}(n)=\left<C_i(n)n+D_i(n)(2n^2+n+1)\right>_q
\end{split}
\end{equation*}
One sees that $Orb(C+D\delta(n))$  depends only  on $C,D$ and $r$. % where $n=qk+r$. 
For the sake of simplicity and from the periodicity of the orbit, we may well denote $Orb(C+D\delta(n))$ by $Orb_{CD}(r)$.
%If we express a set $Orb(C+D\delta(n))$ by $Orb_{CD}(r)$, we obtain that for $n=qk+r$

Now we can express the partial zeta values at $s=0$:
$$\zeta_q(0,(C+D\delta(n))O_{K_n})=A_0(r)+A_1(r)k+\cdots A_d(r)k^d,$$
where 
$$A_i(r)=\sum_{(A,B)\in Orb_{CD}(r)}B^i_{AB}(r),$$
for
\begin{equation*}
\begin{split}
B_{AB}^0(r)&=(\frac{1}{2}-\nu_{AB}^{\Gamma_1(r)}(r))(\nu_{AB}^{\Gamma_1(r)-1}(r)-\frac{1}{2})+\frac{r+1}{6}(6\nu_{AB}^{\Gamma_1(r)}(r)^2-6\nu_{AB}^{\Gamma_1(r)}(r)+1)\\
&+\frac{\tau_1(r)}{12}(6qd_{AB}^0(r)^2-6qd_{AB}^0(r)
+12qd_{AB}^0(r)^2-q) \\
&+\frac{1}{12}(6\nu_{AB}^{\Gamma_1(r)-1}(r)^2-6\nu_{AB}^{\Gamma_1(r)-1}(r)+1)
-\frac{1}{12}(6\nu_{AB}^{0}(r)^2-6\nu_{AB}^{0}(r)+1)\\
&+\frac{1}{2}(\gamma_1(r)-1)d_{AB}^0(r)^2+\frac{1-2d_{AB}^0(r)}{2}(d_{AB}^0(r)\gamma_1(r)-d_{AB}^0(r)\\
&+\nu_{AB}^{0}(r)-\nu_{AB}^{\Gamma_1(r)}(r))-\frac{\gamma_1(r)-1}{12},\\
%\end{split}
%\end{equation*}
%\begin{equation*}
%\begin{split}
B_{AB}^1(r)&=
\frac{q^2}{6}(6\nu_{AB}^{\Gamma_1(r)}(r)^2-6\nu_{AB}^{\Gamma_1(r)}(r)+1)\\&+ 
\frac{q}{12}(6q d_{AB}^0(r)^2+6qd_{AB}^0(r)-12qd_{AB}^0(r)^2-q)
\end{split}
\end{equation*}
and
$$d^0_{AB}(r)=\frac{\left<(2r+1)B+A\right>_q}{q}$$
$$\nu_{AB}^0(r)=\frac{\left<B\right>_q}{q}$$
$$\nu_{AB}^{\Gamma_1(r)}(r)=\frac{\left<(2r^2-r)B+(r-1)A\right>_q}{q}$$
$$\nu_{AB}^{\Gamma_1(r)-1}(r)=\frac{\left<(2r^2-3r-1)B+(r-2)A\right>_q}{q}$$
$$\gamma_1(r)=\left<2r+1\right>_q$$
$$\tau_1(r)=\frac{2r+1-\left<2r+1\right>_q}{q}.$$
%\end{exam}

%\begin{exam}
\subsection{Case 2: $f(n)= 16n^4 + 32n^3 + 24n +3$}\label{case2}
For square free $f(n)=16n^4+32n^3+24n^2+12n+3$, let $K_n:=\FQ(\sqrt{f(n)})$. 
Let us fix $\fb_n = %Then the ring of integer $
O_{K_n}$. If $O_{K_n}$ is $[1,\delta(n)]$ for $\delta(n)$ described as before. %where
%$\delta(n):=\sqrt{f(n)}+[\sqrt{f(n)}]+1.$
For this family, we have: 
\begin{itemize}
\item $\delta(n) =\sqrt{f(n)}+[\sqrt{f(n)}]+1,$
\item $\delta(n)-1= [[8n^2+8n+2, 2n+1]].$
\item The totally positive fundamental unit $\epsilon_n$ is $(2n+1)^3+1+(2n+1)\sqrt{f(n)}$.
\end{itemize}

%Then $\delta(n)-1$ has purely periodic continued fraction $[[8n^2+8n+2, 2n+1]].$ Moreover totally positive fundamental unit $\epsilon_n$ of $K_n$ is $(2n+1)^3+1+(2n+1)\sqrt{f(n)}$.
We can again easily check that $N(C+D\delta(n))$ is invariant modulo $q$ for $n=qk+r$. 
%Moreover we can describe $Orb(C+D\delta(n))$ by the following $\epsilon_n$ action. 

For $C+D\delta(n)\in F_{\delta(n)}$, we have
$$\epsilon_n \ast (C_i(n)+D_i(n)\delta(n))=C_{i+1}(n)+D_{i+1}(n) \delta(n),$$
where
\begin{equation*}
\begin{split}
C_0(n)&=C, \quad D_0(n)=D\\
C_{i+1}(n) &=\langle(64n^5+160n^4+168n^3 + 104n^2+ 38n+7) D_i(n)\\
+& (8n^3 +12n^2+6n+2)C_i(n)\rangle_q\\
%\begin{split}
D_{i+1}(n)&=\left<(16n^3+24n^2+14n+4)D_i(n)+(2n+1)C_i(n) \right>_q
\end{split}
\end{equation*}

Let $n=qk+r$ for $0 \le r < q$. 
%From above we can easily check that $C_i(n)$ and $D_i(n)$ are functions only depending $r, C, D,$ for $n=qk+r.$ 
%I
Denoting  $Orb(C+D\delta(n))$ by $Orb_{CD}(r)$, the partial zeta value at $0$ is %we obtain that for $n=qk+r$
$$\zeta_q(0,(C+D\delta(n))O_{K_n})=A_0(r)+A_1(r)k+\cdots + A_d(r)k^d,$$
where $A_i(r)=\sum_{(A,B)\in Orb_{CD}(r)}B^i_{AB}(r)$,
for
\begin{equation*}
\begin{split}
B_{AB}^0(r)=&(\frac{1}{2}-\nu_{AB}^{\Gamma_1(r)}(r))(\nu_{AB}^{\Gamma_1(r)-1}(r)-\frac{1}{2}) \\ & +\frac{8r^2+8r+6}{12}(6\nu_{AB}^{\Gamma_1(r)}(r)^2-6\nu_{AB}^{\Gamma_1(r)}(r)+1)\\
&+\frac{\tau_1(r)}{12}(6qd_{AB}^0(r)^2-6qd_{AB}^0(r)+12qd_{AB}^0(r)^2-q)\\ 
&+\frac{1}{12}(6\nu_{AB}^{\Gamma_1(r)-1}(r)^2-6\nu_{AB}^{\Gamma_1(r)-1}(r)+1)\\
&-\frac{1}{12}(6\nu_{AB}^{0}(r)^2-6\nu_{AB}^{0}(r)+1)+\frac{1}{2}(\gamma_1(r)-1)d_{AB}^0(r)^2\\
&+\frac{1-2d_{AB}^0(r)}{2}(d_{AB}^0(r)\gamma_1(r)-d_{AB}^0(r)+\nu_{AB}^{0}(r)-\nu_{AB}^{\Gamma_1(r)}(r))\\ &-\frac{\gamma_1(r)-1}{12},\\
%\end{split}
%\end{equation*}
%\begin{equation*}
%\begin{split}
B_{AB}^1(r) = &
\frac{2q^2+4q^2r}{3}(6\nu_{AB}^{\Gamma_1(r)}(r)^2-6\nu_{AB}^{\Gamma_1(r)}(r)+1)\\&+
\frac{q}{6}(6q d_{AB}^0(r)^2+6qd_{AB}^0(r)-12qd_{AB}^0(r)^2-q),\\
B_{AB}^2(r)=&\frac{2q^3}{3}(6\nu_{AB}^{\Gamma_1(r)}(r)^2-6\nu_{AB}^{\Gamma_1(r)}(r)+1),
\end{split}
\end{equation*}
and
\begin{equation*}
\begin{split}
d^0_{AB}(r)&=\frac{\left<(8r^2+8r+3)B+A\right>_q}{q}\\
\nu_{AB}^0(r)&=\frac{\left<B\right>_q}{q}\\
\nu_{AB}^{\Gamma_1(r)}(r)&=\frac{\left<(4r^2+2r+1)B+2rA\right>_q}{q}\\
\nu_{AB}^{\Gamma_1(r)-1}(r)&=\frac{\left<4r^2B+(2r-1)A\right>_q}{q}\\
\gamma_1(r)&=\left<2r+1\right>_q\\
\tau_1(r)&=\frac{2r+1-\left<2r+1\right>_q}{q}.
\end{split}
\end{equation*}
%\end{exam}

\begin{remark}
The family of \ref{case2} has not been touched in literature in the context of class number problem
or other particular problems in arithmetic. Especially, most known families of real quadratic fields,  where class number problems are solved,
are Richaud-Decherd type which is generated by some quadratic polynomials. It would be highly 
interesting if one could answer those questions for other types of families than R-D types. 
\end{remark}

\end{document}